\documentclass{amsart}

\usepackage{amssymb} 

\pdfoutput=1 

\DeclareMathOperator{\supp}{supp}

\DeclareMathOperator{\Cof}{Cof}

\DeclareMathOperator{\forces}{\Vdash}

\newtheorem{theorem}{Theorem}

\newtheorem{remark}{Remark}

\newtheorem{definition}{Definition}

\begin{document}

\author{James Cummings}
\address{Department of Mathematical Sciences\\
         Carnegie Mellon University\\
         Pittsburgh PA 15213-3890\\
         USA} 
\email{jcumming@andrew.cmu.edu}

\author{Mirna D{\v z}amonja}
\address{School of Mathematics\\
         University of East Anglia\\
         Norwich Research Park\\
         Norwich NR4 7TJ\\
         UK}
\email{M.Dzamonja@uea.ac.uk}

\author{Itay Neeman}
\address{Department of Mathematics\\
         University of California, Los Angeles\\
         Los Angeles CA 90095-1555\\
         USA} 
\email{ineeman@math.ucla.edu}

\title{Iteration of strongly $\kappa^+$-cc forcing posets}

\maketitle

\section{Introduction}

   One of the basic results in iterated forcing states that a finite support iteration of ccc forcing is ccc.
   It is natural to look for extensions of this result: the most natural setting for generalisations is to let $\kappa$ be an uncountable
   regular cardinal such that $\kappa^{<\kappa} = \kappa$, and consider $<\kappa$-support iterations in which each iterand is $\kappa$-closed and $\kappa^+$-cc.
   It is known that (even for the case where $\kappa = \aleph_1$ and CH holds) such iterations do not in general have $\kappa^+$-cc \cite{ShelahStanley}, so
   we will need to strengthen the closure and chain condition hypotheses on the iterands.

   Shelah \cite{Shelah} proved that if we strengthen the chain condition assumption a lot and the closure assumption a little then
   we get a useful iteration theorem. More precisely, let $\kappa = \kappa^{<\kappa}$ and say that a poset $\mathbb P$
   is {\em regressively $\kappa^+$-cc} if it enjoys the following property: for every sequence $(p_i)_{i < \kappa^+}$ of conditions in $\mathbb P$ 
   there exist a club set $E \subseteq \kappa^+$ and a regressive function $f$ on  $E \cap \Cof(\kappa)$ such that
   $f(\alpha) = f(\beta)$ implies $p_\alpha$ is compatible with $p_\beta$. This looks technical, but can be motivated by the observation
   that if $\mathbb P$ was proved to be $\kappa^+$-cc by the standard $\Delta$-system and amalgamation arguments then the proof very 
   likely shows that $\mathbb P$ is regressively $\kappa^+$-cc. Shelah's iteration theorem states that a $<\kappa$-support
   iteration with $\kappa$-closed, well met, and regressively $\kappa^+$-cc iterands is regressively $\kappa^+$-cc. Here a  poset
   is {\em well met} if any pair of compatible conditions has a greatest lower bound (glb): Shelah \cite{Shelah} showed that in general
   this technical condition can not be removed. 
   
   We will prove an iteration theorem where the chain condition hypothesis is strengthened in a different direction. 
   Motivation for this work includes some results by Mekler \cite{Mekler} where the ccc is proved  
   using elementary submodels, and the more recent surge of interest (initiated by Mitchell's work
   on $I[\omega_2]$ \cite{Mitchell})  in the notion of strong properness. 

   In Section \ref{background} we give some background on forcing posets, elementary submodels and generic conditions. Section \ref{main} contains the statement and
   proof of our main theorem. Finally Section \ref{further} discusses some generalisations.

\section{Background} \label{background}

    For the rest of this paper we fix  an uncountable regular cardinal such that $\kappa^{<\kappa} = \kappa$.  
    We make the convention that when we write ``$N \prec H_{\theta}$'' we mean ``$N \prec (H_{\theta}, \in, <_\theta)$''
    where $<_\theta$ is a wellordering of $H_\theta$. The structure $(H_{\theta}, \in, <_\theta)$ has definable Skolem functions,
    so that if $N, N' \prec H_\theta$ then $N \cap N' \prec H_\theta$. When $N \prec H_\theta$ we write $\bar N$ for the
    transitive collapse of $N$, $\rho_N : N \simeq \bar N$ for the transitive collapsing map, and
    $\pi_N: \bar N \simeq N$ for its inverse.

\begin{definition} Let $\mathbb Q$ be a forcing poset and let $M \prec H_\theta$. A model $M$ is
  {$\kappa$-good for $\mathbb Q$} if and only if  $\kappa, {\mathbb Q} \in M$, $\vert M \vert = \kappa$
  and ${}^{<\kappa} M \subseteq M$.
\end{definition} 

\begin{remark} If $\mathbb Q \in H_\theta$, then the set of $M$ which are $\kappa$-good for $\mathbb Q$ is
   stationary in $P_{\kappa^+} H_\theta$. 
\end{remark} 

 When $M$ is $\kappa$-good for $\mathbb Q$ and $G$ is $\mathbb Q$-generic over $V$, we will study the subset
 $G \cap M$ of ${\mathbb Q} \cap M$. In a mild abuse of notation we sometimes write
 $\bar G$ for the subset $\rho_M[G \cap M]$ of the poset $\bar {\mathbb Q}$.  We write $M[G]$ for the
 set of elements of form $\dot\tau^G$ where $\dot\tau$ is a $\mathbb Q$-name in $M$.  

\begin{definition} Let $M$ be $\kappa$-good for $\mathbb Q$. Then:

\begin{enumerate}

\item  A condition $q \in {\mathbb Q}$ is {\em $(M, {\mathbb Q})$-generic} iff $q$ forces that
  $\bar G$ is  $\bar {\mathbb Q}$-generic over $\bar M$, and {\em strongly $(M, {\mathbb Q})$-generic} iff it forces that
  $\bar G$ is  $\bar {\mathbb Q}$-generic over $V$. 
 
\item  If $q \in {\mathbb Q}$ and $r \in {\mathbb Q} \cap M$, then $r$ is a {\em strong properness residue of $q$ (for $M$)} 
    iff for every $s \in {\mathbb Q} \cap M$ with $s \le r$, $q$ is compatible with $s$.  We write {\em spr} to abbreviate strong properness residue.

\end{enumerate} 

\end{definition} 


Assume that $M$ is $\kappa$-good for $\mathbb Q$. The following facts are standard: 

\begin{itemize}

\item  If $G$ is ${\mathbb Q}$-generic over $V$, then $M[G] \prec H_\theta[G] = H_\theta^{V[G]}$.
    If in addition $\mathbb Q$ is $\kappa$-closed then $V[G] \models {}^{<\kappa} M[G] \subseteq M[G]$. 

\item 
 A condition $q$ is  $(M, {\mathbb Q})$-generic iff $q$ forces that $M[\dot G] \cap V = M$. In this case
  $q$ forces that $\pi_M$ can be lifted to an elementary embedding $\pi_M: \bar M[\bar G] \rightarrow H_\theta[G]$.  

\item  A condition $q$ is strongly  $(M, {\mathbb Q})$-generic iff the set of conditions in $\mathbb Q$ which have a spr  
   for $M$ is dense below $q$. 

\item  The poset $\mathbb Q$ is $\kappa^+$-cc iff every condition in $\mathbb Q$ is 
   $(M, {\mathbb Q})$-generic. 

\end{itemize}

\begin{definition}
  A forcing poset $\mathbb Q$ is {\em strongly $\kappa^+$-cc} if and only if for all large $\theta$,
  for every  $M \prec H_\theta$ which is $\kappa$-good for $\mathbb Q$, every condition in $\mathbb Q$ is strongly $(M,{\mathbb Q})$-generic. Equivalently, densely many conditions have a spr for $M$, and this implies that in fact all conditions have a spr for $M$.  
\end{definition}



\section{An iteration theorem} \label{main}

\begin{theorem} \label{mainthm}  Let $\kappa$ be uncountable with $\kappa^{<\kappa} = \kappa$.  Let $\mathbb P$ be an iteration with $<\kappa$-supports such that each iterand ${\mathbb Q}_\alpha$
  is forced at stage $\alpha$ to have 
  the following properties:

\begin{enumerate}

\item  ${\mathbb Q}_\alpha$ is strongly $\kappa^+$-cc.

\item \label{cheesy1} ${\mathbb Q}_\alpha$ is well met.     
    
\item  \label{cheesy2} Every directed subset of ${\mathbb Q}_\alpha$ of size less than $\kappa$ has a glb.

\end{enumerate}
   Then $\mathbb P$ is strongly $\kappa^+$-cc. 

\end{theorem}

Depending on the exact way one defines ``directed'' in condition (\ref{cheesy2}), condition (\ref{cheesy2}) may be read to subsume condition (\ref{cheesy1}).

    Before proving the theorem, we digress briefly to illustrate the difficulties and motivate the main idea.
  Consider the case of an iteration ${\mathbb P}_2 = {\mathbb Q}_0 * \dot{\mathbb Q}_1$ of length two, where ${\mathbb Q_0}$ is
  strongly $\kappa^+$-cc and forces that $\dot{\mathbb Q}_1$ is strongly $\kappa^+$-cc. Let
  $M$ be $\kappa$-good for ${\mathbb P}_2$, and let $(q_0, \dot q_1)$ be an arbitrary condition for which we aim to construct a spr.
  If $r_0$ is a spr for $q_0$ and $M$, while $\dot r_1$ names a spr for $\dot q_1$ and $M[\dot G_0]$, then we are not warranted
  in claiming that $(r_0, \dot r_1)$ is a spr for $(q_0, \dot q_1)$. The issue is that while $\dot r_1$ names something which is the denotation
  of a term in $M$, there is no reason to think $\dot r_1$ itself is in $M$. In this simple case we can cope by first extending $q_0$ to
  some $q_0'$, which determines the identity of some term $\dot r_1'$ which denotes a spr for $\dot q_1$, and then choosing $r_0'$
  which is a spr for $q_0'$: this clearly becomes problematic for an iteration of infinite length. We will deal with this kind of problem by     
  building a spr on every relevant coordinate simultaneously. This is similar to the approach taken by \cite{GiltonNeeman}, but without a need for side conditions.

\begin{remark} \label{glbremark}  It is easy to see that condition \ref{cheesy2} is
   preserved by iteration with
   $<\kappa$-supports, so that $\mathbb P$ satisfies it.  To be explicit, if $D$ is a directed subset of
   $\mathbb P$ with $\vert D \vert < \kappa$  then we construct a glb $p$ for $D$ inductively.
   We build $p$ so that $\supp(p) = \bigcup_{t \in D} \supp(t)$: at stage $i$ we have that $p \restriction i$
   is a glb for $\{ t \restriction i : t \in D \}$, observe that $p \restriction i$ forces
    $\{ t(i) : t \in D \}$ to be directed, and choose $p(i)$ to name a glb for this set.  
\end{remark}

\begin{proof}[Proof of Theorem \ref{mainthm}]

 Let the length of the iteration $\mathbb P$ be $\gamma$, let $\theta$ be sufficiently large and let  
 $M$ be $\kappa$-good for $\mathbb P$. 
 Let $p \in {\mathbb P}$ be arbitrary. We will produce $q \le p$ such that $q$ has a spr $r$ for $M$.

 We note that if $\alpha \in M \cap \gamma$ and $G_\alpha$ is ${\mathbb P}_\alpha$-generic,
 then it is routine to check that $M[G_\alpha]$ is $\kappa$-good for ${\mathbb Q}_\alpha$ in $V[G_\alpha]$.
 It follows that every condition in ${\mathbb Q}_\alpha$ has a spr for $M[G_\alpha]$.

 We choose a certain auxiliary model $H$ such that $q, M \in H$ and $\vert H \vert < \kappa$. 
  To construct $H$ we build an increasing chain of models $(H_i)_{i<\omega}$  and a strictly decreasing chain of conditions $(p_i)_{i < \omega}$ in $\mathbb P$ such that:  
\begin{enumerate}
\item  For all $i$, $H_i \prec H_\theta$ and $\vert H_i \vert < \kappa$.
\item   $\supp(p) \cup  \{ p, M \} \subseteq  H_0$. 
\item  $p_0 = p$.
\item For all $i$,  $p_{i+1} \le p_i$ and  $p_{i+1} \in D$ for every dense open $D \in H_i$.  
\item For all $i$,  $H_i \cup \supp(p_{i+1})  \cup  \{ p_{i+1}, H_i \} \subseteq  H_{i+1}$.
\end{enumerate} 
   
 We may choose $p_{i+1}$ because (using Remark \ref{glbremark}) $\mathbb P$ is $\kappa$-closed. 
 At the end we set $H = \bigcup_n H_n$. 
 By Remark \ref{glbremark})  
  the sequence $(p_n)$ has a glb $q$.


 We record some information: 
\begin{enumerate}

\item  By construction $H \prec H_\theta$, $\vert H \vert < \kappa$ and $p, M \in H$.

\item  By  Remark \ref{glbremark}, $\supp(q) = \bigcup_n \supp(p_n)$
  and  $q \restriction \alpha$ forces that $q(\alpha)$ is the glb of the sequence $(p_n(\alpha))$. 

\item If $g = \{ x \in {\mathbb P} \cap H : \exists i \; p_i \le x \}$, then $g$ is a filter on $\mathbb P \cap H$ 
  which meets every dense open set in $H$.

\item By  definition, $q$ is the glb of $g$. We claim that $g = \{ x \in {\mathbb P} \cap H : q \le x \}$. Clearly if $x \in g$ then
  $q \le x$, and if $x \notin g$ then by genericity there is $n$ such that $p_n \perp x$ and so $q \nleq x$. 

\item We claim that the  support of $q$ is $H \cap \gamma$. By construction $\supp(p_n) \subseteq H_n \cap \gamma$ for all $n$, and so
  $\supp(q) \subseteq H \cap \gamma$; conversely if $\alpha \in H \cap \gamma$ then by genericity there is  $n$ such that $\alpha \in \supp(p_n)$.  


\end{enumerate}

The set  $g \cap M$ is a directed subset of $\mathbb P$ and
 $\vert g \cap M \vert \le \vert H \vert < \kappa$, so $g \cap M$  has an glb $r$.
  Since ${}^{<\kappa} M \subseteq M$, $g \cap M \in M$ and so by elementarity  $r \in M$.

\smallskip

\noindent Main Claim:  $r$ is a spr for the condition $q$ and the model $M$.

\smallskip

\noindent Proof of Main Claim:  We let $s \le r$ with $s \in M$ and build inductively a condition $q^*$
  such that $q^*$ is a common refinement of $s$ and $q$. The induction is easy except at
  coordinates $\alpha \in \supp(s) \cap \supp(q)$, so fix such an $\alpha$.  The support of $s$ is contained in $M$,
  and the support of $q$ is contained in $H$,  so  $\alpha \in H \cap M \cap \gamma$.
  Note that $s \le r$ and by induction  $q^* \restriction \alpha \le s \restriction \alpha$, 
  so that $q^* \restriction \alpha \forces s(\alpha) \le r(\alpha)$.

  For each $i < \omega$, define a set $D_i \subseteq {\mathbb P}$ as follows: $D_i$ is the set of
  $t \in \mathbb P$ such that either $t \perp p_i$, or $t \le p_i$ and
  there is $\dot r \in M$ such that $t \restriction \alpha$ forces
  ``$t(\alpha) \le \dot r$, and $\dot r$ is a spr for $p_i(\alpha)$ and $M[\dot G_\alpha]$''.
  Since $\alpha, p_i, M \in H$ we have by elementarity that $D_i \in H$.  

  We claim that $D_i$ is dense. Let $t_0 \in {\mathbb P}$ be arbitrary.
  If $t_0$ is incompatible with $p_i$ then $t_0 \in D_i$, otherwise
  we find $t_1 \le t_0, p_i$.  Extending $t_1 \restriction \alpha$ if necessary, we may assume that
  $t_1 \restriction \alpha$ determines some $\dot r \in M$ which denotes a spr
  for $t_1(\alpha)$; now $t_1 \restriction \alpha$ forces that $\dot r$ and $t_1(\alpha)$ are
  compatible so extending $t_1$ at coordinate $\alpha$ we obtain a condition $t_2 \le t_1$ such that  
  $t_2 \restriction \alpha$ forces $t_2(\alpha) \le \dot r$. Since
  $t_2 \le t_1 \le p_i$ we have $t_2 \restriction \alpha \forces t_1(\alpha) \le p_i(\alpha)$,
  so $t_2 \restriction \alpha$ forces that $\dot r$ is a spr for $p_i(\alpha)$. 
   
   By the construction of the sequence $(p_i)$, we find $j$  such that $p_j \in D_i$. 
   From the definitions $p_j \le p_i$ (that is $j \ge i$), and $p_j \restriction \alpha$  forces
   ``$p_j(\alpha) \le \dot r$ and $\dot r$ is a spr for $p_i(\alpha)$'' for some $\dot r \in M$.
   As $p_j, p_i, \alpha, M \in H$  we may assume by elementarity that $\dot r \in M \cap H$.
   Now if we let $r^*$ be the condition in
   $\mathbb P$ that has $\dot r$ at coordinate $\alpha$ and is otherwise trivial,
   $p_j \le r^* \in M \cap H$ so that $r^* \in g \cap M$.

   So $r \le r^*$, and since $q^* \restriction \alpha \le s \restriction \alpha \le r \restriction \alpha$
   we have $q^* \restriction \alpha \forces r(\alpha) \le r^*(\alpha) = \dot r$.
   Since also $q^* \restriction \alpha \le p_j \restriction \alpha$,
   $q^* \restriction \alpha$ forces that $\dot r$ is a spr for $p_i(\alpha)$.
   Since $q^* \restriction \alpha \forces s(\alpha) \le r(\alpha) \le \dot r$,
   $q^* \restriction \alpha$ forces that $s(\alpha)$ is compatible with $p_i(\alpha)$.

  Now we force below $q^* \restriction \alpha$ to obtain a generic object $G_\alpha$, and work
  in $V[G_\alpha]$ to compute a lower bound for the decreasing sequence $(s(\alpha) \wedge p_i(\alpha))$.
 Let $q^*(\alpha)$ name a lower bound, then
  $q^* \restriction \alpha$ forces that $q^*(\alpha)$ is a lower bound for the sequence
  $(p_i(\alpha))$, and  (since $q^* \restriction \alpha \le q \restriction \alpha$)
  also that $q(\alpha)$ is the glb for the sequence $(p_i(\alpha))$. Hence
  $q^* \restriction \alpha$ forces that $q^*(\alpha) \le q(\alpha)$. Hence
  $q^* \restriction \alpha \forces q^*(\alpha) \le q(\alpha), s(\alpha)$ as required.

\end{proof} 

\section {Further results} \label{further}

With more work we can weaken the closure hypotheses on the iterands as follows:
it is enough to assume that each iterand ${\mathbb Q}_\alpha$ 
is forced to be $<\kappa$-strategically closed, to be countably closed, and to satisfy
the strengthened form of countable strategic closure in which
   move $\omega$ is required to be a glb for the moves played at finite stages. 

   The iteration theorem can also be generalised in other directions. For example
   let $S \subseteq \kappa^+ \cap \Cof(\kappa)$ be stationary,
   and define a poset to be $S$-strongly $\kappa^+$-cc if sprs exist for $\kappa$-good
   models $M$ with $M \cap \kappa^+ \in S$.
   Then $S$-strongly $\kappa^+$-cc forcing posets preserve the stationarity of $S$, and an
   iteration of $S$-strongly
   $\kappa^+$-cc posets with appropriate closure properties is $S$-strongly $\kappa^+$-cc.  
   To prove the generalisation to $S$-strongly $\kappa^+$-cc posets, simply restrict throughout to $M$
   such that $M \cap \kappa^+ \in S$. 

We briefly sketch the proof of the generalisation weakening the closure hypothesis on the iterands. 

We can construct $p_i$ and $q$ as in the proof of Theorem \ref{mainthm} from the weaker hypotheses.
$p_{i+1}$ can be constructed using $<\kappa$-strategic closure. If $\sigma$ is a strategy for player II to
produce descending chains of length $\omega$ with a glb, and taking $\sigma\in H_0$, one can use the fact
that $p_{i+1}$ meets all dense open sets in $H_i$ to find a play $(u_n)_{n<\omega}$ by $\sigma$ so that
$p_{i+1}\leq u_{2i+1}\leq u_{2i}\leq p_i$. This ensures that $(p_i)_{i<\omega}$ has a glb.

The final argument in the proof of Theorem \ref{mainthm}, obtaining a lower bound for the sequence
$(s(\alpha)\land p_i(\alpha))$, goes through with countable closure. 

The only other use of closure in the proof is in defining $r$, a glb for $g\cap M$. We prove that this can
be done with the weakened assumptions. 

The support of $r$ is $M\cap H\cap \gamma$. We work by induction on $\alpha\in M\cap H\cap \gamma$ to
define $r(\alpha)$, assuming that $r\restriction\alpha$ has been defined and is a glb for $(g\cap M)\restriction\alpha$.
Passing to the transitive collapse $\bar{H}$ of $H$, we have that $\bar{g}=\rho_H[g]$ is generic over $\bar{H}$
for $\rho_H({\mathbb P})$.
So $\bar{g}\restriction\alpha$ is generic for $\rho_H({\mathbb P}\restriction\alpha)$ over $\bar{H}$,
and $\bar{g}(\alpha)$ is generic for $\bar{\mathbb Q}_\alpha=\rho_H(\dot{\mathbb Q}_\alpha)[\bar{g}\restriction\alpha]$
over $H[\bar{g}\restriction \alpha]$.

By the strong chain condition, $\bar{g}(\alpha)\cap \rho_H(M)[\bar{g}\restriction\alpha]$ is generic over
$\rho_H(M)[\bar{g}\restriction\alpha]$. Using the strategic closure of the $\alpha^{\rm th}$ iterand
it follows that for each $i$, there is a lower bound
$w_i\in \bar{g}(\alpha)\cap \rho_H(M)[\bar{g}\restriction\alpha]$ for
$\bar{g}(\alpha)\cap \rho_H(M)[\bar{g}\restriction\alpha]\cap \rho_H(H_i)[\bar{g}\restriction\alpha]$.
Let $\dot{\tau}_\alpha\in H\cap M$ be a strategy for player II to produce descending chains of length $\omega$ with a glb
in $\dot{\mathbb Q}_\alpha$. Using the genericity of $\bar{g}(\alpha)\cap \rho_H(M)[\bar{g}\restriction\alpha]$
one can pick $w_i$ to be part of a play by $\rho_H(\dot{\tau}_\alpha)[\bar{g}\restriction\alpha]$.
Let $\dot{w}_i$ name $w_i$. Note that by genericity the fact that the conditions $\dot{w}_i$ are part of a play
by $\rho_H(\dot{\tau}_\alpha)$ is forced by conditions in $\bar{g}(\alpha)\cap \rho_H(M)[\bar{g}\restriction\alpha]$.
Then $r\restriction\alpha$, being a lower bound for $(g\cap M)\restriction\alpha$, forces that
the conditions $\pi_H(\dot{w_i})$ are part of a play according to $\dot{\tau}_\alpha$,
and therefore $(\pi_H(\dot{w}_i))_{i<\omega}$ has a glb. Let $\dot{r}(\alpha)$ name this glb.
One can check that then $r\restriction\alpha+1$ is a glb for $g\restriction\alpha+1$.

\end{document}